\documentclass[11pt]{amsart}

\usepackage{amsmath,amsthm,amssymb,latexsym,amscd}
\usepackage{setspace}
\usepackage[dvips]{graphicx}
\usepackage{longtable} 
\usepackage{subfigure} 
\usepackage{pstricks}  
\usepackage{pst-node}  
\usepackage{scalefnt}  
\usepackage{enumitem}
\usepackage{mdwlist}
\usepackage{tikz}
\usepackage{bm}
\usepackage{thmtools}
\usepackage{thm-restate}
\usepackage{cleveref}
\usepackage[mathscr]{euscript}

\DeclareMathAlphabet{\mathpzc}{OT1}{pzc}{m}{it}

\usetikzlibrary{arrows,scopes, shapes.geometric, calc}

\usepackage[margin=1.25in]{geometry}

\newtheorem{theorem}{Theorem}[section]
\newtheorem{lemma}[theorem]{Lemma}

\newtheorem{proposition}[theorem]{Proposition}

\newtheorem{conjecture}[theorem]{Conjecture}

\theoremstyle{definition}

\theoremstyle{remark}

\numberwithin{equation}{theorem}

\title{Seymour's second-neighborhood conjecture from a different perspective}
\author{Farid Bouya}
\address{Department of Mathematics\\Louisiana State University\\Baton Rouge, LA 70803, USA}
\email{fbouya1@math.lsu.edu}
\author{Bogdan Oporowski}
\address{Department of Mathematics\\Louisiana State University\\Baton Rouge, LA 70803, USA}
\email{bogdan@math.lsu.edu}

\date{\today}

\def\i4c{internally $4$-connected}

\def\int{\mathrm{int}}

\def\le{\leqslant}
\def\ge{\geqslant}

\def\trans{^\intercal}

\def\i4c{internally $4$-connected}

\def\'{$'$}
\renewcommand\hat{\widehat}

\def\i4c{internally 4-connected}

\begin{document}

\begin{abstract}
  Seymour's Second-Neighborhood Conjecture states that every directed graph whose underlying
  graph is simple has at least one vertex $v$ such that the number of vertices of out-distance $2$
  from $v$ is at least as large as the number of vertices of out-distance $1$ from it. 
  We present alternative statements of the conjecture in the language of linear algebra.
\end{abstract}

\maketitle

\section{Introduction and Basic Definitions}

  In this paper, all directed graphs, or digraphs for short, have underlying graphs that are simple,
  that is, with no loops and no multiple edges.
  Let $D$ be a digraph and let $u$ and $v$ be vertices of $D$.
  We write $d(u,v)$ to denote the length of the shortest directed path from $u$ to $v$;
  if no such path exists, then we put $d(u,v) = \infty$.
  Since we focus on vertices of out-distance one or two from a particular vertex $v$ of $D$,
  we set up the following notation.
  \begin{equation*}
    \begin{aligned}[c]
      N^{+}(v) & = \{ u \in V(D) \mid d(v,u) = 1 \}, \\
      N^{++}(v) & = \{ u \in V(D) \mid d(v,u) = 2 \}, \\
      N^{-}(v) & = \{ u \in V(D) \mid d(u,v) = 1 \}, \\
      N^{--}(v) & = \{ u \in V(D) \mid d(u,v) = 2 \}, 
    \end{aligned}
    \qquad
    \begin{aligned}[c]
      d^{+}(v) & = |N^{+}(v)|, \\
      d^{++}(v) & =|N^{++}(v)|, \\
      d^{-}(v) & = |N^{-}(v)|, \\
      d^{--}(v) & = |N^{--}(v)|.
    \end{aligned}
  \end{equation*}
  Each of the symbols defined above may also have a subscript indicating to which digraph it refers.
  Let $\overleftarrow{D}$ be the digraph obtained from $D$ by reversing the direction on all its edges,
  so that $d_{D}^+(v) = d_{\overleftarrow{D}}^-(v)$.
  The original form of Seymour's Second-Neighborhood Conjecture (SNC) is therefore stated as:

  \begin{conjecture}[SNC] \label{snc:orig}
    Every digraph has a vertex $v$ for which $d^{+}(v) \le d^{++}(v)$.
  \end{conjecture}

  We will adopt some of the notation common in linear algebra.
  In particular, $\mathbf{0}$ will denote a vector or a matrix consisting of all zeros,
  and similarly, $\mathbf{1}$ will denote a vector or a matrix consisting of all ones.
  The identity matrix will be denoted by $I$.
  Even though the dimensions of these matrices or vectors will not be stated explicitly, they may be
  easily inferred from the context.

  When vectors are represented in the matrix form, they will be understood as column vectors,
  but to save space, they will be written as transpositions of row vectors.
  Let $\mathbf{u}=(u_1, u_2, \dots, u_n)\trans$ and let $v=(v_1, v_2, \dots, v_n)\trans$.
  When we express a numerical relation between vectors, such as $\mathbf{u} \le\mathbf{v}$,
  we mean that $u_i\le v_i$ for all $i$ in $\{1,2,\dots,n\}$.
  The relations $<$, $\ge$, $>$, and $=$ are understood in a similar way.
  However, the negated relations, such as $\not\le$, $\not<$, $\not\ge$, $\not>$, and $\not=$
  are understood in a different way.
  When we write, for example, $\mathbf{u}\not\le\mathbf{v}$ we mean that $u_i>v_i$ for at
  least one $i$ in $\{1,2\dots,n\}$, and so for vectors with more than one component,
  the inequality $\mathbf{u}\le \mathbf{v}$ is not equivalent to $\mathbf{u}\not>\mathbf{v}$.
  The same idea applies to all other negated relations.

  A \emph{weight function} on a digraph $D$ is a function $w: V(D) \to [0,\infty)$.
  If the vertices of $D$ are enumerated as $v_1$, $v_2$, \ldots,~$v_n$, then we
  can treat $w$ as a vector:
  $\mathbf{w}=[w(v_1), w(v_2), \dots, w(v_n)]\trans$.
  In fact, we will often blur the distinction between the values of a weight function and the
  components of the vector it determines, and write $\mathbf{w}(v)$ instead of $w(v)$.
  We will extend this notation to sets of vertices and write
  $\mathbf{w}(S)$ to mean $\sum_{v\in S}\mathbf{w}(v)$ for a subset $S$ of $V(D)$.

  In order to write SNC in terms of matrices, we define the \emph{second-neighborhood matrix of $D$} as
  an $n\times n$ matrix $S_D$ whose entries are denoted by $s_{ij}$ and defined as follows:
  \begin{align*}
    s_{ij} = 
    \begin{cases}
      1 \qquad & d(v_i,v_j) = 1,\\
      -1 \qquad & d(v_i,v_j) = 2,\\
      0 \qquad & \text{otherwise.}
    \end{cases}
  \end{align*}
  Note that $S_D\trans$ is the second-neighborhood matrix of $\overleftarrow{D}$.

  In this paper, we have adopted main proof techniques from a paper of Fisher~\cite{Fisher}.

 \section{Conjectures}\label{s:conjectures}

  The main purpose of this paper is to present several statements in the language of linear algebra, each of which
  is equivalent to SNC, in the hope that the tools of linear algebra may yield themselves to attacking the conjecture.
  These statements are the following:

  \begin{conjecture}\label{snc:matrix}
    Every digraph $D$ satisfies $S_D \mathbf{1} \not > \mathbf{0}$.
  \end{conjecture}

  \begin{conjecture}\label{snc:weight}
    Every digraph $D$ and every weight vector $\mathbf{w}$ on $D$ satisfy $S_D\mathbf{w}\not>\mathbf{0}$.
  \end{conjecture}

  \begin{conjecture}\label{snc:score-w}
    For every digraph $D$ there is a non-zero weight vector $\mathbf{w}$ with
    $S_D\mathbf{w}\le\mathbf{0}$.
  \end{conjecture}

  \begin{conjecture}\label{snc:score-v}
    For every digraph $D$, there is a vector $\mathbf{v}$ (not necessarily a weight vector)
    with at least one positive component and such that $S_D\mathbf{v}\le\mathbf{0}$.
  \end{conjecture}

  \begin{conjecture}\label{snc:inverse}
    There is no digraph $D$ such that $S_D^{-1}\ge\mathbf{0}$.
  \end{conjecture}

  The first major result of this paper is the following:

  \begin{theorem} \label{t:equiv}
    Conjectures~\ref{snc:orig}, \ref{snc:matrix}, \ref{snc:weight}, \ref{snc:score-w}, \ref{snc:score-v}, and~\ref{snc:inverse} are equivalent.
  \end{theorem}

  Proving some of the equivalences is significantly harder than proving others, and, indeed, some of these statements, such as
  Conjectures~\ref{snc:matrix} and~\ref{snc:weight} play only auxiliary roles in the arguments.
  The proof of this theorem will be presented in a series of propositions in future sections.

  If one, and thus all, of these conjectures fail, the sets of counterexamples may, and, in fact, do differ between some of them.
  When we compare potential counterexamples and use words like ``minimal'' or ``smaller'', we understand them in terms of the
  number of arcs.
  The fact that the sets of minimal counterexamples to Conjectures~\ref{snc:score-w},~\ref{snc:score-v}, and~\ref{snc:inverse} are the same can be easily
  seen from the proofs of the relevant equivalences.
  However, we find surprising the following:

  \begin{theorem}\label{t:conj}
    Every minimal counterexample to Conjecture~\ref{snc:score-w} is smaller than every minimal counterexample to Conjecture~\ref{snc:matrix}.
  \end{theorem}
 
\section{Equivalences}

  We begin by addressing the equivalence of the first pair of the conjectures.
  We state it without proof, as it is evident.

  \begin{proposition}\label{p:orig-matrix}
    Conjectures~\ref{snc:orig} and~\ref{snc:matrix} are equivalent.
  \end{proposition}

  We proceed now to the equivalence of the next pair of conjectures.

  \begin{proposition}\label{p:matrix-weight}
    Conjectures~\ref{snc:matrix} and~\ref{snc:weight} are equivalent.
  \end{proposition}

  \begin{proof}
    It is clear that Conjecture~\ref{snc:weight} implies Conjecture~\ref{snc:matrix}.

    Suppose now that Conjecture~\ref{snc:weight} fails, and so there are a digraph $D$
    and a weight vector $\mathbf{w}$ on $D$ are such that $S_D\mathbf{w}>\mathbf{0}$.
    Since the set of positive rational numbers forms a dense subset of $[0,\infty)$, we may take a weight vector $\mathbf{w}'$
    sufficiently close to $\mathbf{w}$ so that the components of $\mathbf{w}'$
    are rational and positive, and $S_D\mathbf{w}'>\mathbf{0}$.
    By multiplying $\mathbf{w}'$ by a suitable integer, we obtain a weight
    vector $\mathbf{u}$ whose components are positive integers, and such that $S_D\mathbf{u}>\mathbf{0}$.

    We construct a digraph $D^*$ as follows.
    Enumerate the vertices of $D$ as $v_1$, $v_2$,~\dots,~$v_n$, and suppose that
    $\mathbf{u}=[u(v_1),u(v_2),\dots,u(v_n)]\trans$.
    For each $i$ in $\{1,2,\dots,n\}$, let $V_i$ be a set of $u(v_i)$ elements, and let $V(D^*)$ be the
    disjoint union of all $V_i$'s.
    For each directed edge $(v_i,v_j)$ of $D$, put into $D^*$ directed edges from each element of $V_i$ to each element of $V_j$.
    Let $S_{D^*}$ be the second-neighborhood matrix of $D^*$ and note that in the vector $S_{D^*}\mathbf{1}$, the component corresponding
    to a vertex $v$ of $D^*$ that lies in in some $V_i$ is equal to the component of $S_D\mathbf{u}$ corresponding to the vertex
    $v_i$ of $D$.
    Hence $S_{D^*}\mathbf{1}>\mathbf{0}$, and so $D^*$ is a counterexample to Conjecture~\ref{snc:matrix}.
  \end{proof}

  Our proof of the next equivalence will make use of a classical result in linear algebra, known as Farkas' Lemma, which is stated below.

  \begin{theorem}[Farkas' Lemma]\label{t:farkas}
    Let $M$ be an $(m \times n)$-matrix and let $\mathbf{b}$ be an $m$-dimensional vector.
    Then exactly one of the following statements holds.
    \begin{enumerate}
      \item There is an $n$-dimensional vector $\mathbf{x}$ such that $M\mathbf{x}=\mathbf{b}$ and $\mathbf{x}\ge\mathbf{0}$.
      \item There is an $m$-dimensional vector $\mathbf{y}$ such that $M\trans\mathbf{y}\ge\mathbf{0}$ and $\mathbf{b}\trans\mathbf{y}<0$.
    \end{enumerate}
  \end{theorem}

  \begin{proposition}\label{p:weight-score}
    Conjectures~\ref{snc:score-w} and \ref{snc:weight} are equivalent.
    Moreover, a digraph $D$ is a counterexample to Conjecture~\ref{snc:score-w} if and only if $\overleftarrow{D}$
    is a counterexample to Conjecture~\ref{snc:weight}.
  \end{proposition}

  \begin{proof}
    Suppose $D$ is digraph on $n$ vertices.
    Construct a new matrix $M$ with $n+1$ rows and $2n$ columns by assembling together
    smaller matrices, as follows:
    $$M =
      \left[\begin{array}{c|c}
        S_D & I \\
      \hline
      \mathbf{1}\trans & \mathbf{0}\trans 
      \end{array}\right],$$
    and let $\mathbf{b}$ be the $(n+1)$-dimensional standard basis vector $[0,0,\dots,0,1]\trans$.

    For the remainder of the proof, we present a list of statements (1)--(9) that are equivalent to one another.
    It is easy to see that consecutive statements are equivalent, and we remark that the equivalence
    between (5) and (6) follows from Theorem~\ref{t:farkas}.
    \begin{enumerate}
      \item
        Digraph $D$ is a counterexample to Conjecture~\ref{snc:score-w}.
      \item The following system fails for every $n$-dimensional vector $\mathbf{w}$.
        $$\begin{cases}
          S_D \mathbf{w} \le \mathbf{0}\\
          \mathbf{w} \ge \mathbf{0}\\
          \mathbf{w} \neq \mathbf{0}
        \end{cases}$$
      \item The following system fails for every $n$-dimensional vector $\mathbf{u}$.
        $$\begin{cases}
          S_D \mathbf{u} \le \mathbf{0}\\
          \mathbf{u} \ge \mathbf{0}\\
          \mathbf{1}\trans\mathbf{u}=1
        \end{cases}$$
      \item
        The following system fails for every two $n$-dimensional vectors $\mathbf{u}$ and $\mathbf{z}$.
        $$\begin{cases}
          S_D \mathbf{u}+\mathbf{z} = \mathbf{0}\\
          \mathbf{z}\ge \mathbf{0}\\
          \mathbf{u}\ge \mathbf{0}\\
          \mathbf{1}\trans\mathbf{u}=1
        \end{cases}$$
      \item
        The following system fails for every $2n$-dimensional vector $\mathbf{x}$.
        $$\begin{cases}
          M\mathbf{x}=\mathbf{b}\\
          \mathbf{x}\ge\mathbf{0}
        \end{cases}$$
        \item
        There is an $(n+1)$-dimensional vector $\mathbf{y}$ that satisfies the following system.
        $$\begin{cases}
          M\trans \mathbf{y}\ge \mathbf{0}\\
          \mathbf{b}\trans \mathbf{y}<0
        \end{cases}$$
        \item
          There are an $n$-dimensional vector $\mathbf{p}$ and a scalar $r$ that satisfy the following system.
          $$\begin{cases}
            S_D\trans\mathbf{p}+r\mathbf{1}\ge \mathbf{0}\\
            \mathbf{p}\ge\mathbf{0}\\
            r<0
          \end{cases}$$
        \item
          There is an $n$-dimensional vector $\mathbf{p}$ that satisfies the following system.
          $$\begin{cases}
            S_D\trans\mathbf{p} > \mathbf{0}\\
            \mathbf{p}\ge\mathbf{0}\\
          \end{cases}$$
        \item
          Digraph $\overleftarrow{D}$ is a counterexample to Conjecture~\ref{snc:weight}.
    \end{enumerate}
  \end{proof}

  Next we show that Conjectures~\ref{snc:score-v} and~\ref{snc:inverse} are equivalent.

  \begin{proposition} \label{p:sv-in}
    Conjectures \ref{snc:score-v} and \ref{snc:inverse} are equivalent, with the same set of counterexamples.
  \end{proposition}

  \begin{proof}
    Let $D$ be a digraph, and suppose first that the matrix $S_D$ is not invertible.
    Then there is a non-zero vector $\mathbf{u}$ such that $S_D\mathbf{u}=\mathbf{0}$.
    If $\mathbf{u}$ has a positive component, then let $\mathbf{v}=\mathbf{u}$;
    otherwise let $\mathbf{v}=-\mathbf{u}$.
    Then $\mathbf{v}$ testifies to the fact that $D$ satisfies Conjecture~\ref{snc:score-v}.
    Also, $D$ vacuously satisfies Conjecture~\ref{snc:inverse}, and so both conjectures hold
    for digraphs with non-invertible second-neighborhood matrices.

    Suppose now that $S_D$ is invertible,
    and let $\sigma_D$ be the map defined by
    $\sigma_D\colon\mathbf{w}\mapsto S_D\mathbf{w}$.

    Consider the statement:
    \begin{enumerate}
      \item
         Digraph $D$ is a counterexample to Conjecture~(\ref{snc:score-v}).
    \end{enumerate}
    It is equivalent to the statement that
    no vector $\mathbf{w}$ satisfies both $S_D\mathbf{w}\le\mathbf{0}$ and
    $\mathbf{w}\not\le\mathbf{0}$, which, in turn, is equivalent to the statement:
    \begin{enumerate}
      \item[(2)]
      If $\sigma_D(\mathbf{w})\le\mathbf{0}$, then $\mathbf{w}\le\mathbf{0}$.
    \end{enumerate}
    Since $S_D$ is invertible, $\sigma_D$ is bijective and thus has an inverse,
    and so statement~(2) is equivalent to the following:
    \begin{enumerate}
      \item[(3)]
      If $\mathbf{w}\le\mathbf{0}$, then $\sigma_D^{-1}(\mathbf{w})\le\mathbf{0}$.
    \end{enumerate}
    Note that $\sigma_D$ is also linear, and so~(3) is equivalent to the statement:
    \begin{enumerate}
      \item[(4)]
      If $\mathbf{w}\ge\mathbf{0}$, then $\sigma_D^{-1}(\mathbf{w})\ge\mathbf{0}$.
    \end{enumerate}
    Now, we observe that $\sigma_D^{-1}(\mathbf{w})=S_D^{-1}\mathbf{w}$, and so~(4) may
    be restated as:
    \begin{enumerate}
      \item[(5)]
      $S^{-1}_D\mathbf{w}\ge \mathbf{0}$ for every vector $\mathbf{w}\ge \mathbf{0}$.
    \end{enumerate}
    The last statement holds if and only if every entry of $S^{-1}_D$ is non-negative.
  \end{proof} 

The last equivalence is established in the next section.

\section{Counterexamples}

  In this section, we will compare the various sets of potential counterexamples to the conjectures
  discussed in this paper.

  For each $N$ in $\{\ref{snc:matrix}, \ref{snc:score-w}, \ref{snc:score-v}\}$,
  let $\mathscr{X}_N$ denote the set of counterexamples to Conjecture~$N$,
  and let $\overleftarrow{\mathscr{X}}_N=\{\overleftarrow{D} \mid D\in \mathscr{X}_N\}$.
  Intuitively, we may think of each $\overleftarrow{\mathscr{X}}_N$ as the set of
  counterexamples to ``Conjecture~$N$ stated for in-neighbors''.
  
  The first proposition comparing the above sets of counterexamples is an immediate consequence
  of the statements of the conjectures, so it is stated without proof.
  
  \begin{proposition}\label{p:x22-x23}
    $\mathscr{X}_{\ref{snc:score-v}} \subseteq \mathscr{X}_{\ref{snc:score-w}}$ and
    $\overleftarrow{\mathscr{X}}_{\ref{snc:score-v}} \subseteq \overleftarrow{\mathscr{X}}_{\ref{snc:score-w}}$.
  \end{proposition}

  The next proposition is almost as obvious.

  \begin{proposition}\label{p:x21-x'23}
    $\mathscr{X}_{\ref{snc:matrix}} \subseteq \overleftarrow{\mathscr{X}}_{\ref{snc:score-w}}$.
  \end{proposition}

  \begin{proof}
    Suppose $D \in \mathscr{X}_{\ref{snc:matrix}}$.
    It is obvious that $D$ is also a counterexample to Conjecture~\ref{snc:weight}, and
    Proposition~\ref{p:weight-score} asserts that $\overleftarrow{D}$ is a counterexample to Conjecture~\ref{snc:score-w}, as well;
    the conclusion follows.
  \end{proof}

  \begin{lemma} \label{l:x22-x'23}
    Every minimal element of $\mathscr{X}_{\ref{snc:score-w}}$ is a member of $\mathscr{X}_{\ref{snc:inverse}}$.
  \end{lemma}

  \begin{proof}
    Let $D$ be a minimal element of $\mathscr{X}_{\ref{snc:score-w}}$, and let $V(D) = \{v_1 , v_2 , \ldots, v_n \}$.
    By the minimality of $D$, for each $D - v_i $ there is a non-zero non-negative weight vector
    $\mathbf{w}_i$ satisfying $S_{D - v_i}\mathbf{w}_i \le \mathbf{0}$.
    We can extend $\mathbf{w}_i$ to a weight vector $\mathbf{\hat{w}}_i$ on $D$ by putting $\mathbf{\hat{w}}_i(v_i) = 0$. Note that $\mathbf{\hat{w}}_i(N^+(v_j)) - \mathbf{\hat{w}}_i(N^{++}(v_j)) \le 0$ for $j \neq i$.
    If for some $i$, the weight vector $\mathbf{\hat{w}}_i$ satisfies $S_D \mathbf{\hat{w}}_i \le \mathbf{0}$,
    then we reach a contradiction.
    Therefore, we may assume that $\mathbf{\hat{w}}_i(N^+(v_i)) - \mathbf{\hat{w}}_i(N^{++}(v_i)) > 0$ for all $i$.
    Let $\hat{W}$ be the ($n \times n$)-matrix whose $i$th column is $\mathbf{\hat{w}}_i$,
    and let $C=S_D\hat{W}$.
    Then the entries of $C$ may be expressed as
    $c_{ij} = {\mathbf{\hat{w}}}_j(N^+(v_i)) - {\mathbf{\hat{w}}}_j(N^{++}(v_i))$, which implies that 
    $c_{ij}$ is positive if and only if $i=j$.

    We use a process similar to the Gauss-Jordan elimination to turn $C$ into the identity matrix $I_n$.
    The only difference is that we work with columns instead of rows, so we do elementary column operations.
    If we are successful, the identity matrix $I_n$ may be expressed as $C$ multiplied on the right by an appropriate transformation matrix $T$, that is, $I_n=CT$.
    To be more precise, we do the following:
    \begin{enumerate}
      \item Start by putting $i=1$ and $X= (x_{ij})= C$.
      \item If $i>n$, then $X$ is equal to $I_n$. Exit.
      \item If $x_{ii} \le 0$, exit. Otherwise, add suitable multiples of the $i$th column of $X$ to other columns of $X$ to make the
        $i$th row of $X$ zero (except for $x_{ii}$).
      \item Divide the $i$th column by $x_{ii}$.
      \item Add 1 to $i$. Go to $(2)$.
    \end{enumerate}
    If during this process we get non-positive $i$th diagonal (that is, the algorithm exits through step (3) because $x_{ii} \le 0$),
    then a non-negative, non-zero linear combination of
    $S_D\hat{\mathbf{w}}_1$, $S_D\hat{\mathbf{w}}_2$, \dots,~$S_D\hat{\mathbf{w}}_i$ is non-positive, say,
    $$a_1 S_D\hat{\mathbf{w}}_1 +a_2 S_D\hat{\mathbf{w}}_2 + \cdots + a_i S_D\hat{\mathbf{w}}_i \le \mathbf{0}.$$
    This is equivalent to $S_D \left( a_1\hat{\mathbf{w}}_1 +a_2\hat{\mathbf{w}}_2 + \cdots + a_i\hat{\mathbf{w}}_i \right) \le \mathbf{0}$,
    which contradicts the fact that $D \in \mathscr{X}_{\ref{snc:score-w}}$.
    Therefore the procedure described above never results in the matrix $X$ having a non-positive entry on the main diagonal,
    so the algorithm never exits through step~(3), and always exits through step~(2) instead, giving us the identity matrix~$I_n$.
    Note that in this process, we only add non-negative multiples of a column to other columns.
    This means that the elementary matrices associated with the matrix operations are all non-negative, therefore their product $T$ is also non-negative.
    Let $W' = \hat{W}T$, let $\mathbf{w}'_i$ be the $i$th column of $W'$, and let $\mathbf{e}_i$ be the $i$th column of $I_n$, that is, the $i$th $n$-dimensional standard basis vector.
    Then $W'$ is non-negative.
    We have
    $$I_n = CT = S_D \hat{W}T = S_D W'.$$
This means that $S_D$ has non-negative inverse, so $D \in \mathscr{X}_{\ref{snc:inverse}}$, as required.
  \end{proof}

  Now we are ready to provide the last part of the proof of Theorem~\ref{t:equiv}.

  \begin{proposition} \label{p38}
    Conjectures \ref{snc:score-w} and \ref{snc:score-v} are equivalent.
  \end{proposition}

  \begin{proof}
    Clearly, Conjecture~\ref{snc:score-w} implies Conjecture~\ref{snc:score-v}.

    Suppose now that Conjecture~\ref{snc:score-w} fails, and so some digraph $D$ is a minimal element of 
    $\mathscr{X}_{\ref{snc:score-w}}$. 
    Lemma~\ref{l:x22-x'23} implies that $D\in \mathscr{X}_{\ref{snc:inverse}}$,
    so Conjecture~\ref{snc:inverse} fails.
    Proposition ~\ref{p:sv-in} now implies that Conjecture~\ref{snc:score-v} fails as well.
  \end{proof}

  The remainder of the paper is devoted to proving Theorem~\ref{t:conj}.
  Most of the work will be contained in the following:

  \begin{lemma}\label{l:min}
    If a digraph $D$ is a minimal member of $\mathscr{X}_{\ref{snc:score-w}}$,
    then $D\not\in\overleftarrow{\mathscr{X}}_{\ref{snc:matrix}}$.
  \end{lemma}

  \begin{proof}
    Suppose, for a contradiction, that $D$ is a minimal member of $\mathscr{X}_{\ref{snc:score-w}}$
    that also belongs to $\overleftarrow{\mathscr{X}}_{\ref{snc:matrix}}$.
    Since Proposition~\ref{p:x21-x'23} asserts that
    $\overleftarrow{\mathscr{X}}_{\ref{snc:matrix}} \subseteq \mathscr{X}_{\ref{snc:score-w}}$,
    we also have
    \begin{enumerate}
      \item[(1)]
      $D$ is a minimal element of $\overleftarrow{\mathscr{X}}_{\ref{snc:matrix}}$.
    \end{enumerate}

    The minimality of $D$ in $\overleftarrow{\mathscr{X}}_{\ref{snc:matrix}}$ implies that it is strongly connected, and
    the fact that $\overleftarrow{D}$ is a counterexample
    to SNC implies that the minimum in-degree of $D$ is at least two; in fact it is at least seven (see~\cite{KL}).

    Let $y$ be an arbitrary vertex of $D$, let $xy$ be an arc of $D$, and let $D'=D\setminus xy$.
    For a vertex $v$ of $D$, let $a(v)=d^-_D(v)-d^{--}_D(v)$ and let $a'(v)=d^-_{D'}(v)-d^{--}_{D'}(v)$.
    Note that $a(v) \le a'(v)$ whenever $v\neq y$.
    If $D$ has a directed path of length two from $x$ to $y$, then $a'(y)=a(y)-2$;
    otherwise $a'(y)=a(y)-1$.
    We show that
    \begin{enumerate}
      \item[(2)]
      $a'(y) = -1$ and $a'(v) \ge 1$ for $v \neq y$.
    \end{enumerate}

    It is not hard to see that $a(v) \in \{1,2 \}$; see~\cite{Brantner} for a justification.
    This means that $a'(y) \in \{-1,0,1 \}$ and $a'(v) \ge 1$ for $v \neq y$.
    In the case $a'(y) = 1$, we reach a contradiction with the minimality of $D$ in $\overleftarrow{\mathscr{X}}_{\ref{snc:matrix}}$.
    We will show that $a'(y) = 0$ cannot occur either.

    Suppose, for a contradiction, that $a'(y)=0$, and let $z$ be a vertex in $N^-_{D'}(y)$.
    We define a weight vector $\mathbf{u}$ on $D'$ as follows:
    $$\mathbf{u}(v) = 
      \begin{cases}
        1 & \text{if } v \ne z;\text{ and}\\
        \frac{3}{2} & \text{if } v=z;
      \end{cases}
    $$
    Now, we have $S_{D'}\trans\mathbf{u}>\mathbf{0}$, and an argument very similar to the proof of
    Proposition~\ref{p:weight-score} implies
    that $D'$ fails Conjecture~\ref{snc:score-w}, which contradicts the minimality of $D$ in $\mathscr{X}_{\ref{snc:score-w}}$.
    Thus we conclude that $a'(y)=-1$.

    Since $D \in \overleftarrow{\mathscr{X}}_{\ref{snc:matrix}}$, it satisfies $a(y) > 0$, and thus, it must be that $a(y)=1$,
    in other words,
    \begin{enumerate}
      \item[(3)]
      $d^-_D(y) = d^{--}_D(y)+1$.
    \end{enumerate}

    Let $y$ be a vertex of $D$ with the largest possible in-degree~$d$.
    Then~(3) implies that $d^{--}(y)=d-1$.
    Let $N^-(y)=X=\{x_1,x_2,\dots,x_d\}$ and let $N^{--}(y)=Z=\{z_1,z_2,\dots,z_{d-1}\}$.
    Consider the digraph $D'=D\setminus x_1y$ and note that the minimality of $D$ implies
    that there is a weight vector $\mathbf{w}'$ such that $S_{D'}\mathbf{w}'\le\mathbf{0}$.
    The last inequality is equivalent to stating that
    $\mathbf{w}'(N^+(u))\le\mathbf{w}'(N^{++}(u))$ for every vertex $u$ of $D'$,
    which, in turn, implies that
    \begin{equation*}
      \sum_{u\in V(D')}\mathbf{w}'(N^+_{D'}(u)) \le \sum_{u\in V(D')}\mathbf{w}'(N^{++}_{D'}(u))
    \end{equation*}
    Note that $\mathbf{w}'(u)$ appears $d^-_{D'}(u)$ times on the left side of the above inequality,
    while it appears $d^{--}_{D'}(u)$ times on the right side.
    By~(3), we have $d^-_{D'}(u) \ge d^{--}_{D'}(u)+1$ whenever $u\neq y$ and $d^-_{D'}(y) = d^{--}_{D'}(y)-1$,
    and so $\mathbf{w}'(y)\ge \mathbf{w}'(V\setminus\{y\})$.
    If $\mathbf{w}'(y) > \mathbf{w}'(V\setminus\{y\})$,
    then $\mathbf{w}'(N^{+}_{D'}(x_2)) > \mathbf{w}'(N^{++}_{D'}(x_2))$, which is impossible.
    It follows that $\mathbf{w}'(y) = \mathbf{w}'(V\setminus\{y\})$, which implies that
    $\mathbf{w}'(N^+_{D'}(u)) = \mathbf{w}'(N^{++}_{D'}(u))$ for every vertex $u$ of $D'$.

    Let $S = \{u\in V(D'): u\ne y \text{ and } \mathbf{w}'(u)>0\}$, and observe that
    $\mathbf{w}'(S) = \mathbf{w}'(y)$.
    Let $k=\mathbf{w}'(y)$, let $Z'=Z\cup \{x_1\}$, and let $X'=X\setminus \{x_1\}$.
    By construction, $N^{--}_{D'}(y)=Z'$, and so $y\in N^{++}_{D'}(z)$ for every $z\in Z'$.
    This implies that $\mathbf{w}'(N^{++}_{D'}(z))\ge k$, and, further, that
    $\mathbf{w}'(N^{++}_{D'}(z))=k=\mathbf{w}'(N^+(z))$.
    This means that $D$ has an arc $zs$ for every $z\in Z'$ and every $s\in S$.
    Since $y$ has the largest in-degree in $D$, we have $Z' = N^-(s)$ for every $s\in S$.

    Note that $y\not\in S$, and if $X'\cap S$ had an element $x$, then we would have
    $\mathbf{w}'(N^{++}(x))<k\le\mathbf{w}'(N^{+}(x))$, which is impossible;
    hence $X'\cap S=\varnothing$.
    Similarly, $Z'\cap S=\varnothing$.
    Therefore $( \{y \} \cup N^-(y) \cup N^{--}(y)) \cap S = \varnothing$.
    Since $D$ is strongly connected and $S$ is non-empty, $D$ has a vertex $t$ of in-distance three from~$y$,
    which, clearly, is in neither $X$ nor $Z$.
    Since all vertices in $S$ have $d$ in-neighbors in $X\cup Z$, there is no arc in $D$ in the form $ts$ with $s\in S$,
    and so $\mathbf{w}'(N^{+}(t))=0$.
    But every $z \in N^{--}(y)$ has arcs to all members of $S$, so $\mathbf{w}'(N^{++}(t)) \ge k$; a~contradiction.
  \end{proof}

  Finally, we are ready to prove Theorem~\ref{t:conj}

  \begin{proof}[Proof of Theorem~\ref{t:conj}]
    Suppose $D$ is a minimal counterexample to Conjecture~\ref{snc:matrix}.
    Then $\overleftarrow{D} \in \overleftarrow{\mathscr{X}}_{\ref{snc:matrix}}$.
    Lemma~\ref{p:x21-x'23} implies that
    $\overleftarrow{\mathscr{X}}_{\ref{snc:matrix}} \subseteq \mathscr{X}_{\ref{snc:score-w}}$, and so
    $\overleftarrow{D}$ is also a counterexample to Conjecture~\ref{snc:score-w}.
    If $\overleftarrow{D}$ were minimal, then Lemma~\ref{l:min} would imply that
    $D\notin\mathscr{X}_{\ref{snc:matrix}}$, which would be a contradiction.
  \end{proof}

\end{document}